%% file: main.tex
\documentclass[letterpaper, 10 pt, conference]{ieeeconf}  

\IEEEoverridecommandlockouts                              
\usepackage{additionalSetting}
\usepackage{anyfontsize}



\title{\LARGE \bf
Synthesis of constrained robust feedback policies and model predictive control
}

\author{Dennis Gramlich, Carsten W. Scherer, Hannah Häring and Christian Ebenbauer
\thanks{The second author thanks Deutsche Forschungsgemeinschaft (DFG, German Research Foundation) under Germany's Excellence Strategy - EXC 2075 – 390740016 for funding and acknowledges the support by the Stuttgart Center for Simulation Science (SimTech).}%
\thanks{Dennis Gramlich, Hannah Häring and Christian Ebenbauer are with the Chair of Intelligent Control Systems,
        RWTH-Aachen,
        D-52074 Aachen, Germany
        {\tt\small \{dennis.gramlich, hannah.haering ,christian.ebenbauer\} @ic.rwth-aachen.de}}%
\thanks{Carsten W. Scherer is with the Chair of Mathematical Systems Theory, University of Stuttgart, 70569 Stuttgart, Germany {\tt\small carsten.scherer@imng.uni-stuttgart.de}}%
}%

\definecolor{mycolor1}{rgb}{0.00000,0.44700,0.74100}%
\definecolor{mycolor2}{rgb}{0.85000,0.32500,0.09800}%
\definecolor{mycolor3}{rgb}{0.92900,0.69400,0.12500}%
\definecolor{mycolor4}{rgb}{0.49400,0.18400,0.55600}%
\definecolor{mycolor5}{rgb}{0.46600,0.67400,0.18800}%
\definecolor{mycolor6}{rgb}{0.30100,0.74500,0.93300}%
\definecolor{mycolor7}{rgb}{0.63500,0.07800,0.18400}%

\begin{document}

\maketitle
\thispagestyle{empty}
\pagestyle{empty}

\begin{abstract}

In this work, we develop a method based on robust control techniques to synthesize robust time-varying state-feedback policies for finite, infinite, and receding horizon control problems subject to convex quadratic state and input constraints. To ensure constraint satisfaction of our policy, we employ (initial state)-to-peak gain techniques. Based on this idea, we formulate linear matrix inequality conditions, which are simultaneously convex in the parameters of an affine control policy, a Lyapunov function along the trajectory and multiplier variables for the uncertainties in a time-varying linear fractional transformation model. In our experiments this approach is less conservative than standard tube-based robust model predictive control methods.

\end{abstract}

\section{INTRODUCTION}
\label{sec:1}

In this paper, we deal with the synthesis of robust policies subject to linear time-varying discrete-time dynamics and convex constraints. Solving this type of control problem has applications, for example, in robust model predictive control (MPC) and trajectory planning.

Constraint satisfaction and robustness to uncertainties in dynamical systems are among the most important design objectives in controller synthesis. Controller design methods addressing these objectives are typically ascribed to the fields of optimal and robust control.
A key issue in this setting is the complexity of the policy even for linear quadratic control problems if these are subject to constraints \cite{wen2009analytical} or uncertainties \cite{ramirez2002characterization}. In fact, no method is known today to solve these problems exactly with polynomial complexity (over an infinite horizon). For this reason, relaxation techniques are employed to find suboptimal solutions. Specifically, robust control relies on multiplier relaxations \cite{chou1999stability,scherer2006lmi,veenman2016robust} to guarantee robustness against uncertainties and, e.g., MPC relaxes the infinite horizon optimal control problem by approximating it with a finite horizon problem.

Given these facts, a natural approach to address control problems with constraints and uncertainties is to combine both relaxation techniques, i.e., to study the finite horizon robust control problem subject to constraints. This is the topic of the present paper. To avoid overly conservative stability and feasibility results~\cite{mayne2000constrained}, we optimize over affine linear feedback policies instead of feedforward policies. This is by now generally accepted as the preferred approach for problems with uncertainties \cite{goulart2006optimization}. Specifically, to deal with uncertainties, we utilize multipliers from robust control. To ensure robust constraint satisfaction, we define the constraints as outputs of our system and use the energy-to-peak gain \cite{rotea1993generalized} to bound the maximum of these outputs for a given initial state. The resulting optimization problem can be convexified in all variables as a linear matrix inequality (LMI) problem.

The fusion of robust control with receding horizon optimal control has been studied since at least 1987~\cite{campo1987robust}. From this perspective, the present paper follows the ideas outlined by Kothare in \cite{kothare1994robust,kothare1996robust}, who proposed to solve time-invariant robust control problems subject to constraints online, taking the current state into account. The time-invariance introduces conservatism, which is why \cite{casavola2000min,casavola2004robust,schuurmans2000robust} proposed to additionally optimize over a $N$-step input sequence using Kothares MPC scheme as terminal ingredient. Convex optimization of time-varying feedback policies subject to constraints and uncertainties is then established by minimax MPC as described, e.g., in \cite{lofberg2003minimax}. All these methods rely on solving LMI optimization problems online, which is often considered costly. On the other hand, as we show in \cite{gramlich2023structure} that the structure of LMIs for robust controller synthesis can be exploited to develop faster solvers, which is particularly interesting for online applications like MPC.


After minimax MPC, so-called tube-based MPC methods with fixed feedback terms~\cite{langson2004robust} gained much attention due to their seemingly lower computational complexity. However, this reduced computational effort comes at the price of not being able to optimize the tube and tube controller online and tightening the constraints, resulting in conservatism. Furthermore, optimizing robust performance and incorporating model uncertainty descriptions from robust control is challenging in tube-based MPC. In recent publications, the integration of classical robust control tools, like dynamic IQCs \cite{morato2023stabilizing,schwenkel2022model}, into MPC is explored. In addition, optimizing tube parameters online to reduce conservatism has gained interest \cite{6170547}.

In parallel to robust MPC, finite horizon robust controller design is investigated from a robust control perspective \cite{uchida1992finite,buch2021finite,gramlich2022robust} using LMI-based tools. Constraints have been incorporated recently into such finite horizon robust control settings under the umbrella of reachability analysis \cite{reynolds2021funnel}. This synthesis technique is extended to joint optimization over feedforward and feedback terms in \cite{kim2022joint} for Lipschitz continuous uncertain nonlinearities.

Table \ref{tab:compareMPC} gives an overview of robust MPC schemes from the literature. We indicate whether the MPC methods are recursively feasible (RF), can incorporate linear fractional transformations (LFTs) as uncertainty models, optimize over feedback policies online (OFPO), involve an LMI optimization problem, and how the number of decision variables grows with the prediction horizon. Integrating LFTs is beneficial as these representations of uncertain systems are capable of handling rational dependencies of system matrices on uncertain parameters. Additionally, optimizing feedback policies online is desirable, since this increases the robustness of MPC schemes. LMI optimization problems should be avoided, because the typical online optimization problem in linear MPC is a less costly to solve quadratically constrained quadratic program. Finally, it is desirable that the number of decision variables grows at most linearly in $N$. 

Among the multitude of tube-based MPC methods in the literature, some allow for flexible uncertainty descriptions such as LFTs. However, the offline calculation of the feedback terms leads to unavoidable conservatism. This is in contrast to the other methods considered, which perform these calculations online. Here, LMI-MPC \cite{lofberg2003minimax} and SLSMPC can be considered as disturbance feedback MPC methods, since LMI-MPC uses disturbance feedback and SLSMPC optimizes the closed loop transfer function from disturbances to outputs. These methods can achieve any of the desired properties in Table \ref{tab:compareMPC}. However, we mention that LFTs can only be handled when incorporating LMIs. Finally, we considered algorithms \cite{kothare1996robust}, which solve time-invariant controller design problems online. Through this, feedback terms are optimized online, LFTs can be considered, and the optimization is recursively feasible. Constraints are taken into account using the (initial state)-to-peak gain. Here, conservatism arises from the assumption of quadratic Value/Lyapunov functions and the time invariance of the controller law. The present approach generalizes \cite{kothare1996robust} to the time-varying case, removing the conservatism caused by time-invariant feedback policies. Compared to disturbance feedback approaches, the quadratic value function still poses a source of conservativeness, but, on the other hand, recursive feasibility is much easier to obtain and the number of decision variables grows only linearly with the prediction horizon. Moreover, LMI-based robust controller design is embedded in our approach.

We start our exposition with Section \ref{sec:2} demonstrating how the energy-to-peak gain enables the integration of constraints in LMI formulations of robust control. The problem formulation of Section \ref{sec:2} does not consider uncertainties in which case we establish that our convexification is feasible at an initial state $\bar{x}$, whenever the corresponding open loop optimal control problem is feasible at $\bar{x}$. Uncertainties are included in Section \ref{sec:3}. We introduce another relaxation step in Section \ref{sec:4} to also handle infinite horizon robust optimal control problems with constraints. Finally, in Section \ref{sec:5} this formulation is shown to be stable and recursively feasible if used in a receding horizon fashion.


\begin{table}
	\caption[Comparison of existing approaches for robust Model Predictive Control.]{This table compares the feastures recursive feasibility (RF), linear fractional representation uncertainty model (LFT), online feedback policy optimization (OFPO), LMI optimization problem (LMIs) and the asymptotic number of decision variables in the prediction horizon $N$ of the MPC schemes \emph{tube MPC} \cite{langson2004robust,lorenzen2019robust,lu2019robust,kohler2019linear}, \emph{$\infty$-hor. MPC} \cite{kothare1996robust,chen2006moving}, SLSMPC \cite{chen2022robust}, DFMPC \cite{goulart2006optimization} (disturbance feedback MPC), LMI-MinMax MPC \cite{lofberg2003minimax} and our approach.}
	\centering
	\begin{tabular}{@{}l|lllllll@{}}
		\toprule
		method & RF & LFTs & OFPO & LMIs & \# dec. vars.
		\\\midrule
		Tube MPC & Yes & (Yes) & No & No & $O(N)$\\
		$\infty$-hor. MPC & Yes & Yes & Yes & Yes & -\\
		SLSMPC & No & No & Yes & No & $O(N^2)$\\
		DFMPC & (Yes) & (Yes) & Yes & (Yes) & $O(N^2)$\\
		LMI-MPC & Yes & Yes & Yes & Yes & $O(N^2)$\\
		Our method & Yes & Yes & Yes & Yes & $O(N)$
		\\\bottomrule
	\end{tabular}
	\label{tab:compareMPC}
\end{table}


\section{The uncertainty-free finite-horizon case}
\label{sec:2}

In a first step, we restrict ourselves to time-varying affine systems without uncertainties of the form
\begin{align}
	x_{k+1} = f_k + A_k x_k + B_k^1 u_k, \label{eq:undisturbedSys}
\end{align}
where $x_k \in \bbR^n$ is the state and $u_k \in \bbR^m$ is the control input and $f_k \in \bbR^n$ is the affine term. We further assume that the control policy satisfies some quadratic constraints, i.e., $u_k$ should be chosen such that
\begin{align}
	v_{ki}^\top v_{ki} &\leq 1 & \forall i &= 1,\ldots,s,\quad k = 0,\ldots,N \label{eq:inputConstraint}
\end{align}
holds for affine outputs $v_{ki} = g_{ki}^2 + C_{ki}^2 x_k + D_{ki}^{21} u_k$. Since affine terms $g_{ki}^2$ are included, these outputs can describe a large class of convex quadratic constraints.

For systems \eqref{eq:undisturbedSys} with constraints \eqref{eq:inputConstraint}, an initial state $\bar{x}$, and a positive definite matrix $P_f$, we study the problem
\begin{align}
	\minimize_{(u_k)_{k=0}^{N - 1}} ~&~ \sum_{k=0}^{N - 1} y_k^\top y_k
	+
	\begin{pmatrix}
		1\\
		x_{N}
	\end{pmatrix}^\top
	P_f
	\begin{pmatrix}
		1\\
		x_{N}
	\end{pmatrix} \label{eq:DP1}\\
	\mathrm{s.t.} ~&~ x_{k+1} = f_k + A_k x_k + B_k^1 u_k, \nonumber\\
	~&~ y_k = g_k^1 + C_k^1 x_k + D_k^{11} u_k, \nonumber\\
	~&~ v_{ki} = g_{ki}^2 + C_{ki}^2 x_k + D_{ki}^{21} u_k, ~~~ i = 1,\ldots,s, \nonumber\\
	~&~ v_{ki}^\top v_{ki} \leq 1, \hspace{28mm} i = 1,\ldots,s, \nonumber\\
	~&~ x_0 = \bar{x}. \nonumber
\end{align}

To approach this optimal control problem, we construct a family of functions $V_k : \bbR^n \to \bbR$ with
\begin{align}
	V_k(x_k) = \begin{pmatrix}
		1\\
		x_k
	\end{pmatrix}^\top 
	\underbrace{\begin{pmatrix}
			p_k^{11} & p_k^{12}\\
			p_k^{21} & P_k^{22}
	\end{pmatrix}}_{=:P_k}
	\begin{pmatrix}
		1\\
		x_k
	\end{pmatrix} \label{eq:LyapParametrization}
\end{align}
defined by a sequence of symmetric matrices $P_k \in \bbR^{(1+n)\times (1+n)}$, a sequence of affine linear control policies $u_k = \pi_k(x_k) = k_k^1 + K_k^2 x_k = K_k \begin{pmatrix}
	1 & x_k^\top
\end{pmatrix}^\top$, and a constant $\nu \in \bbR$ satisfying the conditions
\begin{subequations}
	\label{eq:LyapConditions1}
	\begin{align}
		V_k(x) &\geq  y^\top y + 
		V_{k+1}(x^+), \label{eq:costBound}\\
		\mathrm{for} ~&~ x^+ = f_k + B_k^1 k_k^1 + (A_k + B_k^1K_k^2)x, \nonumber\\
		~&~ y = g_k^1 + D_k^{11} k_k^1 + (C_k^1 + D_k^{11} K_k^2)x,\nonumber\\
		V_0(\bar{x}) &\leq \nu, \label{eq:levelSetBound}\\
		V_{N} (x) &\geq \begin{pmatrix}
			1\\
			x
		\end{pmatrix}^\top
		P_f
		\begin{pmatrix}
			1\\
			x
		\end{pmatrix}, \label{eq:terminalConstraint1} \\
		V_k(x) &\leq \nu \Rightarrow v_{i}^\top v_{i} \leq 1, \label{eq:feasibility}\\
		\mathrm{for} ~&~ v_{i} = g_{ki}^2 + D_{ki}^{21}k_k^1 + (C_{ki}^2 + D_{ki}^{21} K_k^2) x, \nonumber
	\end{align}
\end{subequations}
for all $x \in \bbR^n$, $k=0,\ldots, N-1$ and $i = 1,\ldots,s$. Condition \eqref{eq:costBound} states that $V_k$ decreases at least by the cost $y_k^\top y_k$ at every time-step, \eqref{eq:levelSetBound} signifies that the initial state is contained in a $\nu$-sublevel set of $V_0$, \eqref{eq:terminalConstraint1} implies that $V_{N}$ upper bounds the terminal cost. Lastly, \eqref{eq:feasibility} means that states in the $\nu$-sublevel set of $V_k$ satisfy all constraints.

We can formulate \eqref{eq:LyapConditions1} as LMI constraints in $(P_k)_{k=0}^{N}$ and $\nu$.

\begin{proposition}
	\label{prop:PreliminaryInequaities}
	Let the functions $V_k$ be parametrized as in \eqref{eq:LyapParametrization} and define $\Sigma_0 := \begin{pmatrix}
		1 & \bar{x}^\top
	\end{pmatrix}^\top \begin{pmatrix}
		1 & \bar{x}^\top
	\end{pmatrix}$ as well as
	\begin{align*}
		\begin{pmatrix}
			f_k^K & A_k^K\\ 
			g_k^{1K} & C_k^{1K}\\
			g_{ki}^{2K} & C_{ki}^{2K}
		\end{pmatrix}
		:=
		\begin{pmatrix}
			f_k & A_k & B_k^1\\
			g_k^1 & C_k^1 & D_k^{11}\\
			g_{ki}^2 & C_{ki}^2 & D_{ki}^{21}
		\end{pmatrix}
		\begin{pmatrix}
			1 & 0\\
			0 & I\\
			k_k^1 & K_k^2
		\end{pmatrix}.
	\end{align*}
	Then, \eqref{eq:costBound}-\eqref{eq:terminalConstraint1} are equivalent to the conditions
	\begin{subequations}
	\begin{align}
		0&\succeq
		(\bullet)^\top 
		\begin{pmatrix}
			-p_k^{11} & -p_k^{12}\\
			-p_k^{21} & -P_k^{22}\\
			& & p_{k+1}^{11} & p_{k+1}^{12}\\
			& & p_{k+1}^{21} & P_{k+1}^{22}\\
			& & & & I
		\end{pmatrix}
		\begin{pmatrix}
			1 & 0\\
			0 & I\\
			1 & 0\\
			f_k^K & A_k^K\\ 
			g_k^{1K} & C_k^{1K}
		\end{pmatrix}, \label{eq:bellmanInequality1}\\
		\nu &\geq \trace P_0 \Sigma_0, \label{eq:initialInequality1}\\
		P_{N} &\succeq P_f \label{eq:terminalInequality1}
		\intertext{respectively. Moreover, \eqref{eq:feasibility} is implied by the inequality}
		P_k & \succeq \nu \begin{pmatrix}
		g_{ki}^{2K} & C_{ki}^{2K}
	\end{pmatrix}^\top
	\begin{pmatrix}
		g_{ki}^{2K} & C_{ki}^{2K}
	\end{pmatrix}. \label{eq:constraintInequality1}
	\end{align}
	\end{subequations}
\end{proposition}
\begin{proof}
	Left-multiplication with $\begin{pmatrix}
		1 & x^\top
	\end{pmatrix}$ and right-multiplication with $\begin{pmatrix}
	1 & x^\top
\end{pmatrix}^\top$ of the matrix inequalities \eqref{eq:bellmanInequality1} and \eqref{eq:terminalInequality1} reveals the equivalence to \eqref{eq:costBound} and \eqref{eq:terminalConstraint1}, respectively. The same argument shows \eqref{eq:constraintInequality1} $\Rightarrow$ \eqref{eq:feasibility}. Lastly, plugging the definition of $\Sigma_0$ into \eqref{eq:initialInequality1} proves the equivalence to \eqref{eq:levelSetBound}.
\end{proof}

\begin{remark}
	Using the lossless $S$-procedure, it is possible to construct an equivalent LMI constraint for \eqref{eq:feasibility} instead of \eqref{eq:constraintInequality1}. However, we could not convexify the lossless constraint in the all parameters.
\end{remark}

The matrix inequalities of Proposition \ref{prop:PreliminaryInequaities} are linear in the Lyapunov function matrices $(P_k)_{k=0}^{N}$ and the variable $\nu$. In addition, since we impose the constraint $\nu \geq \trace P_0 \Sigma_0$, we obtain the sequence of estimates
\begin{align*}
	\nu &\geq  V_0(x_0) \geq y_0^\top y_0 +  V_1(x_1)\\
	&\geq y_0^\top y_0 + y_1^\top y_1 +  V_2(x_2)\\
	&\geq \sum_{k=0}^{N-1} y_k^\top y_k + V_{N}(x_{N})\\
	&\geq \sum_{k=0}^{N - 1} y_k^\top y_k
	+
	\begin{pmatrix}
		1\\
		x_{N}
	\end{pmatrix}^\top
	P_f
	\begin{pmatrix}
		1\\
		x_{N}
	\end{pmatrix}.
\end{align*}
Thus, by minimizing $\nu$ as objective function, we can try to find a smallest upper bound on the performance of our controller $(K_k)_{k=0}^{N-1}$ while certifying constraint satisfaction. However, the matrix inequalities of Proposition~\ref{prop:PreliminaryInequaities} are not linear in the controller $(K_k)_{k=0}^{N-1}$ such that they cannot be utilized for convex synthesis. For this reason, we provide a convexification in all variables in the following theorem.

\begin{theorem}
	\label{thm:PeakGainControllerSynthesis}
	Define the decision variables
	\begin{align*}
		\widetilde{P}_k &:= P_k^{-1}, \hspace{2mm} \widetilde{K}_k := K_k P_k^{-1}, \hspace{2mm} \tilde{\nu} := \nu^{-1}, \hspace{2mm} Z := \nu^{-2} \sqrt{\Sigma_0} P_0 \sqrt{\Sigma_0}
	\end{align*}
	including the slack variable $Z$ and introduce the notation
	\begin{align*}
		\begin{pmatrix}
			\tilde{f}_k^K & \widetilde{A}_k^K\\
			\tilde{g}_k^{1K} & \widetilde{C}_k^{1K}\\
			\tilde{g}_{ki}^{2K} & \widetilde{C}_{ki}^{2K}
		\end{pmatrix}
		:=
		\begin{pmatrix}
			f_k & A_k & B_k^1\\
			g_k^1 & C_k^1 & D_k^{11}\\
			g_{ki}^2 & C_{ki}^2 & D_{ki}^2
		\end{pmatrix}
		\begin{pmatrix}
			\tilde{p}_k^{11} & \tilde{p}_k^{12}\\
			\tilde{p}_k^{21} & \widetilde{P}_k^{22}\\
			\tilde{k}_k^1 & \widetilde{K}_k^2
		\end{pmatrix}.
	\end{align*}
	Then, for positive definite matrices $\widetilde{P}_k$ and $P_k$, \eqref{eq:bellmanInequality1}-\eqref{eq:constraintInequality1} are equivalent to
	\begin{subequations}
	\begin{align}
		0 &\preceq
		\begin{pmatrix}
			\tilde{p}_{k+1}^{11} & \tilde{p}_{k+1}^{12} & 0 & \tilde{p}_k^{11} & \tilde{p}_k^{12}\\
			\tilde{p}_{k+1}^{21} & \widetilde{P}_{k+1}^{22} & 0 & \tilde{f}_k^K & \widetilde{A}_k^K\\
			0 & 0 & I & \tilde{g}_k^{1K} & \widetilde{C}_k^{1K}\\
			\tilde{p}_k^{11} & (\tilde{f}_k^K)^\top & (\tilde{g}_k^{1K})^\top & \tilde{p}_k^{11} & \tilde{p}_k^{12}\\
			\tilde{p}_k^{21} & (\widetilde{A}_k^K)^\top & (\widetilde{C}_k^{1K})^\top  & \tilde{p}_k^{21} & \widetilde{P}_k^{22}
		\end{pmatrix},
		\label{eq:linearCostInequality}\\
		0 &\preceq
		\begin{pmatrix}
			\widetilde{P}_0 & \tilde{\nu} \sqrt{\Sigma_0}\\
			\tilde{\nu} \sqrt{\Sigma_0}^\top & Z
		\end{pmatrix},~~~ \trace Z \leq \tilde{\nu}, \label{eq:linearInitialInequality}\\
		\widetilde{P}_{N} & \preceq P_f^{-1}, \label{eq:linearTerminalInequality}\\
		0 & \preceq
		\begin{pmatrix}
			\tilde{p}_k^{11} & \tilde{p}_k^{12} & \tilde{g}_{ki}^{2K}\\
			\tilde{p}_k^{21} & \widetilde{P}_k^{22} & \widetilde{C}_{ki}^{2K} \\
			\tilde{g}_{ki}^{2K} & \widetilde{C}_{ki}^{2K} & \tilde{\nu} I
		\end{pmatrix}. \label{eq:constraintInequality2}
	\end{align}
	\end{subequations}
	
\end{theorem}

\begin{proof}
	Multiplying \eqref{eq:bellmanInequality1} from both sides by $\widetilde{P}_k$ yields
	\begin{align*}
		0&\succeq
		(\bullet)^\top 
		\begin{pmatrix}
			-\tilde{p}_k^{11} & -\tilde{p}_k^{12}\\
			-\tilde{p}_k^{21} & -\widetilde{P}_k^{22}\\
			& & p_{k+1}^{11} & p_{k+1}^{12}\\
			& & p_{k+1}^{21} & P_{k+1}^{22}\\
			& & & & I
		\end{pmatrix}
		\begin{pmatrix}
			1 & 0\\
			0 & I\\
			\tilde{p}_k^{11} & \tilde{p}_k^{12}\\
			\tilde{f}_k^K & \widetilde{A}_k^K\\
			\tilde{g}_k^{1K} & \widetilde{C}_k^{1K}
		\end{pmatrix}.
	\end{align*}
	To arrive at \eqref{eq:linearCostInequality}, we apply the Schur complement. To infer \eqref{eq:constraintInequality2} from \eqref{eq:constraintInequality1} we proceed analogously. We multiply from both sides by $\widetilde{P}_k$ to obtain
	\begin{align*}
		\widetilde{P}_k \succeq \nu
		\begin{pmatrix}
			\tilde{g}_{ki}^{2K} & \widetilde{C}_{ki}^{2K}
		\end{pmatrix}^\top \begin{pmatrix}
			\tilde{g}_{ki}^{2K} & \widetilde{C}_{ki}^{2K}
		\end{pmatrix}.
	\end{align*}
	Now, applying the Schur complement lemma yields the LMI \eqref{eq:constraintInequality2}. Next, consider the condition $\nu \geq \trace P_0 \Sigma_0 = \trace \widetilde{P}_0^{-1} \Sigma_0$. To handle this constraint, we add the slack variable $Z \succeq \nu^{-2} \sqrt{\Sigma_0} \widetilde{P}_0^{-1} \sqrt{\Sigma_0}$. With this slack variable, we can replace the original constraint by
	\begin{align*}
		\tilde{\nu} &\geq \trace Z, & Z \succeq \tilde{\nu}^2 \sqrt{\Sigma_0} \widetilde{P}_0^{-1} \sqrt{\Sigma_0}.
	\end{align*}
	By taking the Schur complement, the above is equivalent to \eqref{eq:linearInitialInequality}. Finally, to convexify the constraint $P_{N} \succeq  P_f$, we invert the matrices on both sides leading to \eqref{eq:linearTerminalInequality}.
\end{proof}

With Theorem \ref{thm:PeakGainControllerSynthesis} we arrive at the convex relaxation of problem \eqref{eq:DP1} optimized over control policies
\begin{align}
	\maximize_{\widetilde{P}_k,\widetilde{K}_k,\tilde{\nu},Z} ~&~ \tilde{\nu} \label{eq:relaxedDP1}\\
	\mathrm{s.t.} ~&~ \eqref{eq:linearCostInequality} \text{ for } k = 0,\ldots,N-1, \nonumber\\
	~&~ \eqref{eq:constraintInequality2} \text{ for } k=0,\ldots, N-1, \nonumber\\
	~&~ \eqref{eq:linearInitialInequality}, \text{ and } \eqref{eq:linearTerminalInequality}. \nonumber
\end{align}
In general, such a convex relaxation provides only an upper bound on the optimal value of the considered problem. In the present case without uncertainties, we can nevertheless show that strict feasibility of \eqref{eq:DP1} implies strict feasibility of \eqref{eq:relaxedDP1}.

\begin{theorem} \label{thm:losslessCase}
	There exists an input sequence $(u_k)_{k=0}^{N-1}$ such that the constraints $v_{ki}^\top v_{ki} \leq 1$ of \eqref{eq:DP1} are strictly satisfied, if and only if \eqref{eq:relaxedDP1} is strictly feasible.
\end{theorem}

\begin{proof}
	See Appendix \ref{app:1}.
\end{proof}

The costs of \eqref{eq:DP1} and \eqref{eq:relaxedDP1} are not necessarily equivalent, which we discuss in Appendix \ref{app:2}.

\section{The finite-horizon case with uncertainty}
\label{sec:3}

Building on Section \ref{sec:2}, we can incorporate an additional disturbance input in our model and work with
\begin{align}
	x_{k+1} = f_k + A_k x_k + B_k^1 u_k + B_k^2 w_k. \label{eq:uncertainSystem}
\end{align}
In this setup, $w_k \in \bbR^l$ is the disturbance input, which is introduced in addition to the state $x_k \in \bbR^n$ and the input $u_k\in \bbR^m$. We assume that both the disturbances and the controller are subject to constraints, i.e., the control input needs to enforce \eqref{eq:inputConstraint}. Moreover, we require the disturbance input to satisfy the quadratic constraint
\begin{align*}
	\begin{pmatrix}
		z_k\\
		w_k
	\end{pmatrix}^\top \underbrace{\begin{pmatrix}
			M_k^{11} & M_k^{12}\\
			M_k^{21} & M_k^{22}
	\end{pmatrix}}_{=:M_k}
	\begin{pmatrix}
		z_k\\
		w_k
	\end{pmatrix} \geq 0
\end{align*}
for all elements $M_k$ of a convex cone $\calM$, where $z_k = g_k^3 + C_k^3 x_k + D_k^{31} u_k + D_k^{32} w_k$ is some output of our system. We further assume that $M_k^{-1}$ exists for all $M_k\in \calM$, that the inverse is contained in another convex cone $\calM'$ and that for all $M_k \in \calM$ the block $M_k^{11}$ is positive definite and $M_k^{22}$ is negative definite. Introducing quadratically constrained inputs $w_k$ in this fashion is standard in robust control and can be utilized to describe simple bounded disturbances such as $\|w_k\| \leq 1$, but also parametric uncertainties modeled by LFTs as, e.g., in \cite{green2012linear}.

In the following dynamic program, we treat the disturbance as an adversary, i.e., we maximize over $w_k$ while minimizing over $u_k$ in the problem
\begin{align}
	\abrmin_{u_0 \in \calU_0} & \abrmax_{w_0 \in \calW_0} \abrmin_{u_1 \in \calU_1} \abrmax_{w_1 \in \calW_1} \ldots ~ \sum_{k=0}^{N - 1} y_k^\top y_k
	+
	\begin{pmatrix}
		1\\
		x_{N}
	\end{pmatrix}^\top
	P_f
	\begin{pmatrix}
		1\\
		x_{N}
	\end{pmatrix} \label{eq:robustDP}\\
	\mathrm{s.t.} ~&~ x_{k+1} = f_k + A_k x_k + B_k^1 u_k + B_k^2 w_k, \nonumber\\
	~&~ y_k = g_k^1 + C_k^1 x_k + D_k^{11} u_k + D_k^{12} w_k, \nonumber\\
	~&~ v_{ki} = g_{ki}^2 + C_{ki}^2 x_k + D_{ki}^{21} u_k, \hspace{15mm} i = 1,\ldots,s,\nonumber\\
	~&~ z_k = g_k^3 + C_k^3 x_k + D_k^{31} u_k + D_k^{32} w_k, \nonumber\\
	~&~ x_0 = \bar{x}, \nonumber
\end{align}
with sets $\calU_k = \calU_k(x_k)$ and $\calW_k = \calW_k(z_k)$ defined by
\begin{align}
	u_k \in \calU_k(x_k) &\Leftrightarrow v_{ki}^\top v_{ki} \leq 1 & & i = 1,\ldots,s, \label{eq:controlConstraint}\\
	w_k \in \calW_k(z_k) & \Leftrightarrow \begin{pmatrix}
		z_k\\
		w_k
	\end{pmatrix} M_{k} \begin{pmatrix}
		z_k\\
		w_k
	\end{pmatrix} \leq 0 & & \forall M_k \in \calM . \label{eq:disturbanceConstraint}
\end{align}
Note that while minimizing over $u_k$ and maximizing over $w_k$ may seem like a symmetric situation, this is not the case from an optimization perspective. In particular, since the cost function is convex, minimizing over $(u_k)_{k=0}^{N-1}$ for fixed $(w_k)_{k=0}^{N-1}$ is a convex quadratically constrained quadratic program. On the other hand, maximizing over $(w_k)_{k=0}^{N-1}$ for fixed $(u_k)_{k=0}^{N-1}$ is a non-convex quadratically constrained quadratic program. For this reason, we cannot carry over the result of Theorem~\ref{thm:losslessCase} from the uncertainty-free case.

Due to this asymmetry we treat constraints on $u_k$ differently from constraints on $w_k$. Concretely, to handle \eqref{eq:controlConstraint}, we once again search for functions $V_k$ whose sublevel sets guarantee the fullfilment of $u_k \in \calU_k(x_k)$, while \eqref{eq:disturbanceConstraint} is taken into account through the $S$-procedure. Hence, to solve \eqref{eq:robustDP}, we again construct a family of functions $(V_k)_{k=0}^N$ as in \eqref{eq:LyapParametrization}
and a control policy $\pi_k(x_k) = k_k^1 + K_k^2 x_k = K_k \begin{pmatrix}
	1 & x_k^\top
\end{pmatrix}^\top$ satisfying the inequalities \eqref{eq:levelSetBound}-\eqref{eq:feasibility}. In contrast to the disturbance-free case, however, instead of \eqref{eq:costBound}, the increment condition
\begin{align}
	V_k(x) &\geq y^\top y
	+
	V_{k+1}(x^+), \label{eq:robustCostBound}\\
	\mathrm{for} ~&~ x^+ = f_k + B_k^1 k_k^1 + (A_k + B_k^1K_k^2)x + B_k^2 w, \nonumber\\
	~&~ y = g_k + D_k^{11} k_k^1 + (C_k^1 + D_k^{11} K_k^2)x + D_k^{12} w, \nonumber
\end{align}
needs to be satisfied robustly with respect to $w$, i.e., for all $w \in \calW_k(z_k)$ and all $x \in \bbR^n$. By the $S$-procedure, if there exists some $M_k \in \calM$ such that for all $x\in \bbR^n$, $w\in \bbR^l$
\begin{align}
	V_k(x) &\geq y^\top y + \begin{pmatrix}
		z\\
		w
	\end{pmatrix}^\top M_k
	\begin{pmatrix}
		z\\
		w
	\end{pmatrix}
	+
	V_{k+1}(x^+), \label{eq:relaxedRobustCostBound}\\
	\mathrm{for} ~&~ x^+ = f_k + B_k^1 k_k^1 + (A_k + B_k^1K_k^2)x + B_k^2 w, \nonumber\\
	~&~ y = g_k^1 + D_k^{11} k_k^1 + (C_k^1 + D_k^{11} K_k^2)x + D_k^{12} w, \nonumber\\
	~&~ z = g_k^3 + D_k^{31}k_k^1 + (C_k^3 + D_k^{31} K_k^2) x + D_k^{32} w\nonumber
\end{align}
holds, then this implies \eqref{eq:robustCostBound}. Consequently, we can certify robust performance and robust constraint satisfaction using \eqref{eq:relaxedRobustCostBound} and \eqref{eq:levelSetBound}-\eqref{eq:feasibility} as follows.

\begin{proposition}
	\label{prop:PreliminaryRobustInequaities}
	Let the functions $V_k$ be parametrized as in \eqref{eq:LyapParametrization} and define 	
	\begin{align*}
		\begin{pmatrix}
			f_k^K & A_k^K\\ 
			g_k^{1K} & C_k^{1K}\\
			g_k^{3K} & C_{k}^{3K}
		\end{pmatrix}
		=
		\begin{pmatrix}
			f_k & A_k & B_k^1\\
			g_k^1 & C_k^1 & D_k^{11}\\
			g_{k}^3 & C_{k}^3 & D_{k}^{31}
		\end{pmatrix}
		\begin{pmatrix}
			1 & 0\\
			0 & I\\
			k_k & K_k
		\end{pmatrix}.
	\end{align*}
	Then, \eqref{eq:relaxedRobustCostBound} is characterized by
	\begin{align}
		&\resizebox{\linewidth}{!}{$
		(\bullet)^\top
		\begin{pmatrix}
			-p_k^{11} & -p_k^{12}\\
			-p_k^{21} & - P_k^{22}\\
			& & M_k^{22} & & & & M_k^{21}\\
			& & & p_{k+1}^{11} & p_{k+1}^{12}\\
			& & & p_{k+1}^{21} & P_{k+1}^{22}\\
			& & & & & I\\
			& & M_k^{12} & & & & M_k^{11}
		\end{pmatrix}
		\begin{pmatrix}
			1 & 0 & 0\\
			f_k^K & A_k^K & B_k^2\\ 
			g_k^{1K} & C_k^{1K} & D_k^{12}\\
			g_k^{3K} & C_k^{3K} & D_k^{32}\\
			1 & 0 & 0\\
			0 & I & 0\\
			0 & 0 & I
		\end{pmatrix}$}\nonumber\\
		&\hspace{65mm} \preceq 0. \label{eq:robustCostInequality}
	\end{align}
	Furthermore, the conditions \eqref{eq:robustCostInequality} and \eqref{eq:initialInequality1}-\eqref{eq:constraintInequality1} imply robust performance in the sense that $\nu$ upper bounds the optimal cost of \eqref{eq:robustDP} and robust constraint satisfaction in the sense that the policy $(\pi_k)_{k=0}^{N-1}$ chooses feasible inputs for any realization of $(w_k)_{k=0}^{N-1}$ satisfying \eqref{eq:disturbanceConstraint}.
\end{proposition}

\begin{proof}
	To establish equivalence between \eqref{eq:robustCostInequality} and \eqref{eq:relaxedRobustCostBound}, left-multiply $\begin{pmatrix}
		1 & x^\top & w^\top 
	\end{pmatrix}$ and right-multiply its transpose to \eqref{eq:robustCostInequality}. Summing \eqref{eq:relaxedRobustCostBound} for $k = 0,\ldots,N-1$ with $x = x_k$ and $w = w_k$ chosen as a trajectory of \eqref{eq:uncertainSystem} yields
	\begin{align}
		V_0(x_0) \geq \sum_{k=0}^{N-1} \left(y_k^\top y_k + \begin{pmatrix}
			z_k\\
			w_k
		\end{pmatrix}^\top M_k
		\begin{pmatrix}
			z_k\\
			w_k
		\end{pmatrix} \right) + V_{N}(x_{N}). \label{eq:robustPerfCond}
	\end{align}
	Here, the terms $\begin{pmatrix}
		z_k\\
		w_k
	\end{pmatrix}^\top M_k
	\begin{pmatrix}
		z_k\\
		w_k
	\end{pmatrix}$ are positive for all admissible values of $w_k$ and, hence, can be omitted without violating \eqref{eq:robustPerfCond}. This establishes the inequalities
	\begin{align*}
		V_0(x_0) \geq \sum_{k=0}^{N-1} y_k^\top y_k + V_{N}(x_{N}) ~~ \text{and} ~~ V_k(x_k) \geq V_{k+1}(x_{k+1}).
	\end{align*}
	In a next step we use the relations in Proposition \ref{prop:PreliminaryInequaities}. From \eqref{eq:initialInequality1} we conclude $V_0(x_0) \leq \nu$ and from \eqref{eq:terminalInequality1} we infer
	\begin{align*}
		\nu \geq \sum_{k=0}^{N-1}  y_k^\top y_k + \begin{pmatrix}
			1\\
			x_{N}
		\end{pmatrix}^\top
		P_f
		\begin{pmatrix}
			1\\
			x_{N}
		\end{pmatrix},
	\end{align*}
	showing the robust performance bound of $\nu$ and that $x_k^\top P_k x_k \leq \nu$ holds for all $k$. The latter fact together with \eqref{eq:constraintInequality1} implies robust constraint satisfaction.
\end{proof}

As in the disturbance-free case, we are interested in matrix inequalities, which are also linear in the controller parameter. This is provided in the next theorem.

\begin{theorem}
	\label{thm:RobustPeakGainControllerSynthesis}
	Define the decision variables
	\begin{align*}
		\widetilde{P}_k &:= P_k^{-1}, \hspace{2mm} \widetilde{K}_k := K_k P_k^{-1}, \hspace{2mm} \tilde{\nu} := \nu^{-1}, \hspace{2mm} Z := \nu^{-2} \sqrt{\Sigma_0} P_0 \sqrt{\Sigma_0}
	\end{align*}
	including the slack variable $Z$, denote the matrix block
	\begin{align*}
		\begin{pmatrix}
			0 & 1 & 0 & 0 & 0\\
			B_k^2 & 0 & I & 0 & 0\\ 
			D_k^{12} & 0 & 0 & I & 0\\
			D_k^{32} & 0 & 0 & 0 & I
		\end{pmatrix}
		\begin{pmatrix}
			\widetilde{M}_k^{22} & & & & \widetilde{M}_k^{21}\\
			& \tilde{p}_{k+1}^{11} & \tilde{p}_{k+1}^{12}\\
			& \tilde{p}_{k+1}^{21} & \widetilde{P}_{k+1}^{22}\\
			& & & I\\
			\widetilde{M}_k^{12} & & & & \widetilde{M}_k^{11}
		\end{pmatrix}
		(\bullet)^\top
	\end{align*}
	by $\widetilde{Q}_k$ and introduce the notation
	\begin{align*}
		\begin{pmatrix}
			\tilde{f}_k^K & \widetilde{A}_k^K\\
			\tilde{g}_k^{1K} & \widetilde{C}_k^{1K}\\
			\tilde{g}_k^{3K} & \widetilde{C}_k^{3K}
		\end{pmatrix}
		:=
		\begin{pmatrix}
			f_k & A_k & B_k^1\\
			g_k^1 & C_k^1 & D_k^{11}\\
			g_k^3 & C_k^3 & D_k^{31}
		\end{pmatrix}
		\begin{pmatrix}
			\tilde{p}_k^{11} & \tilde{p}_k^{12}\\
			\tilde{p}_k^{21} & \widetilde{P}_k^{22}\\
			\tilde{k}_k^1 & \widetilde{K}_k^2
		\end{pmatrix}.
	\end{align*}
	Then the matrix inequalities \eqref{eq:robustCostInequality}, \eqref{eq:initialInequality1}, \eqref{eq:terminalInequality1} and \eqref{eq:constraintInequality1} are equivalent to the linear matrix inequalities
	\begin{align}
		\begin{pmatrix}
			\tilde{q}_k^{11} & \tilde{q}_k^{12} & \tilde{q}_k^{13} & \tilde{q}_k^{14} & \tilde{p}_k^{11} & \tilde{p}_k^{12}\\
			\tilde{q}_k^{21} & \widetilde{Q}_k^{22} & \widetilde{Q}_k^{23} & \widetilde{Q}_k^{24} & \tilde{f}_k^K & \widetilde{A}_k^K\\
			\tilde{q}_k^{31} & \widetilde{Q}_k^{32} & \widetilde{Q}_k^{33} & \widetilde{Q}_k^{34} & \tilde{g}_k^{1K} & \widetilde{C}_k^{1K}\\
			\tilde{q}_k^{41} & \widetilde{Q}_k^{42} & \widetilde{Q}_k^{43} & \widetilde{Q}_k^{44} & \tilde{g}_k^{3K} & \widetilde{C}_k^{3K}\\
			\tilde{p}_k^{11} & (\tilde{f}_k^K)^\top & (\tilde{g}_k^{1K})^\top & (\tilde{g}_k^{3K})^\top & \tilde{p}_k^{11} & \tilde{p}_k^{12}\\
			\tilde{p}_k^{21} & (\widetilde{A}_k^K)^\top & (\widetilde{C}_k^{1K})^\top & (\widetilde{C}_k^{3K})^\top & \tilde{p}_k^{21} & \widetilde{P}_k^{22}
		\end{pmatrix}\succeq 0,
		\label{eq:robustLinearCostInequality}
	\end{align}
	\eqref{eq:linearInitialInequality}, \eqref{eq:linearTerminalInequality} and \eqref{eq:constraintInequality2} for positive definite $P_k$ and $\widetilde{P}_k$.
\end{theorem}

\begin{proof}
	Applying the dualization lemma (Lemma 4, \cite{iwasaki1998well}) to \eqref{eq:robustCostInequality} proves that \eqref{eq:robustCostInequality} holds if and only if
	\begin{align*}
		\resizebox{\linewidth}{!}{$
			\begin{pmatrix}
				1 & 0 & 0 & 0 & 1 & 0 & 0\\
				0 & I & 0 & 0 & f_k^K & A_k^K & B_k^2\\ 
				0 & 0 & I & 0 & g_k^{1K} & C_k^{1K} & D_k^{12}\\
				0 & 0 & 0 & I & g_k^{2K} & C_k^{2K} & D_k^{22}
			\end{pmatrix}
			\begin{pmatrix}
				-\tilde{p}_k^{11} & -\tilde{p}_k^{12}\\
				- \tilde{p}_k^{21} & - \widetilde{P}_k^{22}\\
				& & \widetilde{M}_k^{22} & & & & \widetilde{M}_k^{21}\\
				& & & \tilde{p}_{k+1}^{11} & \tilde{p}_{k+1}^{12}\\
				& & & \tilde{p}_{k+1}^{21} & \widetilde{P}_{k+1}^{22}\\
				& & & & & I\\
				& & \widetilde{M}_k^{12} & & & & \widetilde{M}_k^{11}
			\end{pmatrix}
			(\bullet)$}
	\end{align*}
	is positive semi-definite. By invoking the Schur complement we obtain \eqref{eq:robustLinearCostInequality}. The remaining matrix inequalities are identical to those in Theorem \ref{thm:PeakGainControllerSynthesis} and proven analogously.
\end{proof}

With this theorem, we can perform LMI synthesis of robust state-feedback controllers subject to constraints for finite horizon problems by solving the SDP
\begin{align}
	\maximize_{\widetilde{P}_k,\widetilde{K}_k,\tilde{\nu},Z,\widetilde{M}_k \in \calM'} ~&~ \tilde{\nu} \label{eq:relaxedRobustDP}\\
	\mathrm{s.t.} ~&~ \eqref{eq:robustLinearCostInequality} \text{ for } k = 0,\ldots,N-1, \nonumber\\
	~&~ \eqref{eq:constraintInequality2} \text{ for } k=0,\ldots, N -1, \nonumber\\
	~&~ \eqref{eq:linearInitialInequality} \text{ and } \eqref{eq:linearTerminalInequality}. \nonumber
\end{align}

\section{The infinite horizon case with uncertainty}
\label{sec:4}

Now, we consider the infinite horizon problem 
\begin{align}
	\abrmin_{u_0 \in \calU_0} \abrmax_{w_0 \in \calW_0} & \abrmin_{u_{1} \in \calU_{1}} \abrmax _{w_{1} \in \calW_{1}} \ldots ~ \sum_{k=0}^{\infty} y_k^\top y_k \label{eq:robustinfiniteDP}\\
	\mathrm{s.t.} ~&~ x_{k+1} = f_k + A_k x_k + B_k^1 u_k + B_k^2 w_k, \nonumber\\
	~&~ y_k = g_k^1 + C_k^1 x_k + D_k^{11} u_k + D_k^{12} w_k, \nonumber\\
	~&~ v_{ki} = g_{ki}^2 + C_{ki}^2 x_k + D_{ki}^{21} u_k, \qquad i = 1,\ldots,s,\nonumber\\
	~&~ z_k = g_k^2 + C_k^2 x_k + D_k^{21} u_k + D_k^{22} w_k, \nonumber\\
	~&~ x_0 = \bar{x}, \nonumber
\end{align}
where the constraint $u_k \in \calU_k(x_k)$ is described by \eqref{eq:controlConstraint} and the constraint $w_k \in \calW_k(z_k)$ is described by \eqref{eq:disturbanceConstraint}. To obtain a tractable relaxation of this problem, we assume that for $k \geq N \in \bbN_0$ the problem parameters are not changing in $k$. Denoting the parameters at time $k$ by
\begin{align*}
	G_k :=
	\begin{pmatrix}
		f_k & A_k & B_k^1 & B_k^2\\
		g_k^1 & C_k^1 & D_k^{11} & D_k^{12} \\
		g_{k i}^2 & C_{k i}^2 & D_{k i}^{21} & 0\\
		g_k^3 & C_k^3 & D_k^{31} & D_k^{32}
	\end{pmatrix},
\end{align*}
this means $G_k = G_{N} ~ \forall k \geq N$. The assumption of constant problem parameters for $k \geq N$ can be well justified, e.g., if the data $(G_k)_{k=0}^\infty$ results from the linearization of a nonlinear system around a reference trajectory, which is at an equilibrium for $k \geq N$. Accordingly, our strategy for solving \eqref{eq:robustinfiniteDP} consists of searching for a function parametrized by $(P_k)_{k=0}^\infty$ and a controller parametrized by $(K_k)_{k=0}^\infty$ which are stationary for $k \geq N$, i.e.,
\begin{align*}
	(P_0,P_1,P_2,\ldots) &= (P_0,\ldots,P_{N-1},P_{N},P_{N},P_{N},\ldots),\\
	(K_0,K_1,K_2,\ldots) &= (K_0,\ldots,K_{N-1},K_{N},K_{N},K_{N},\ldots).
\end{align*}
In the infinite horizon case we need to require \eqref{eq:relaxedRobustCostBound} and \eqref{eq:levelSetBound} for all $k \in \bbN_{\geq N}$ instead of enforcing only the terminal condition \eqref{eq:terminalConstraint1}. To render this requirement tractable, we exploit the stationarity assumption on $P_k$ and $K_k$ and reduce this infinite family to the two constraints
\begin{align}
	V_{N}(x) &\geq y^\top y + \begin{pmatrix}
		z\\
		w
	\end{pmatrix}^\top M_{N}
	\begin{pmatrix}
		z\\
		w
	\end{pmatrix}
	+
	V_{N}(x^+) \label{eq:relaxedRobustTailCost}\\
	\mathrm{for} ~&~ x^+ = f_{N} + B_{N}^1 k_{N}^1 + (A_{N} + B_{N}^1K_{N}^2)x + B_{N}^2 w, \nonumber\\
	~&~ y = g_{N}^1 + D_{N}^{11} k_{N}^1 + (C_{N}^1 + D_{N}^{11} K_{N}^2)x + D_{N}^{12} w, \nonumber\\
	~&~ z = g_{N}^3 + D_{N}^{31}k_{N}^1 (C_{N}^3 + D_{N}^{31} K_{N}) x + D_{N}^{32} w,\nonumber\\
	V_{N}(x) &\leq \nu \Rightarrow v_{i}^\top v_{i} \leq 1, \label{eq:feasibilityForTail}\\
	\mathrm{for} ~&~ v_{i} = g_{{N} i}^2 + D_{{N} i}^2k_{N} + (C_{{N} i}^2 + D_{{N} i}^2 K_{N}) x, \nonumber
\end{align}
for all $x \in \bbR^n$, $w \in \bbR^l$. The constraint \eqref{eq:feasibilityForTail} corresponds to \eqref{eq:feasibility} for $k = N$. However, \eqref{eq:relaxedRobustTailCost} differs from \eqref{eq:robustCostBound}, since the same function $V_{N}$ appears on both sides of the inequality. Nonetheless, we can still express \eqref{eq:relaxedRobustTailCost} as matrix inequality and convexify it, as shown in the following result.

\begin{proposition}
	If the functions $V_k$ are parametrized as in \eqref{eq:LyapParametrization} with stationary sequences as above, then the conditions \eqref{eq:relaxedRobustCostBound},\eqref{eq:feasibility} are satisfied for all $k \in \bbN_0$, if \eqref{eq:robustCostInequality} holds for $k = 0,\ldots,N-1$, \eqref{eq:constraintInequality1} holds for $k = 0,\ldots,N$, and $P_{N}$, $M_{N}$ satisfy the matrix inequality
	\begin{align}
		&\resizebox{\linewidth}{!}{$
			(\bullet)^\top
			\begin{pmatrix}
				-p_{N}^{11} & -p_{N}^{12}\\
				- p_{N}^{21} & - P_{N}^{22}\\
				& & M_{N}^{22} & & & & M_{N}^{21}\\
				& & & p_{N}^{11} & p_{N}^{12}\\
				& & & p_{N}^{21} & P_{N}^{22}\\
				& & & & & I\\
				& & M_{N}^{12} & & & & M_{N}^{11}
			\end{pmatrix}
			\begin{pmatrix}
				1 & 0 & 0\\
				0 & I & 0\\
				0 & 0 & I\\
				1 & 0 & 0\\
				f_{N}^K & A_{N}^K & B_{N}^2\\ 
				g_{N}^{1K} & C_{N}^{1K} & D_{N}^{12}\\
				g_{N}^{3K} & C_{N}^{3K} & D_{N}^{32}
			\end{pmatrix}$}\nonumber\\
		&\hspace{65mm} \preceq 0. \label{eq:infiniteCostInequality}
	\end{align}
\end{proposition}
\begin{proof}
	The proof relies on the same arguments as the proof of Proposition \ref{prop:PreliminaryRobustInequaities} and is thus omitted.
\end{proof}

Essentially, the matrix inequality \eqref{eq:infiniteCostInequality} replaces the terminal condition \eqref{eq:terminalInequality1}. It ensures an upper bound on the infinite tail cost and robust constraint satisfaction for $k \geq N$. Next, we derive a linear formulation of this matrix inequality providing a convex relaxation of \eqref{eq:robustinfiniteDP}.

\begin{theorem}
	\label{thm:infiniteRobustPeakGainControllerSynthesis}
	Define the decision variables
	\begin{align*}
		\widetilde{P}_k &:= P_k^{-1}, \hspace{2mm} \widetilde{K}_k := K_k P_k^{-1}, \hspace{2mm} \tilde{\nu} := \nu^{-1}, \hspace{2mm} Z := \nu^{-2} \sqrt{\Sigma_0} P_0 \sqrt{\Sigma_0}
	\end{align*}
	including the slack variable $Z$, denote the matrix block
	\begin{align*}
		\begin{pmatrix}
			0 & 1 & 0 & 0 & 0\\
			B_{N}^2 & 0 & I & 0 & 0\\ 
			D_{N}^{12} & 0 & 0 & I & 0\\
			D_{N}^{22} & 0 & 0 & 0 & I
		\end{pmatrix}
		\begin{pmatrix}
			\widetilde{M}_{N}^{22} & & & & \widetilde{M}_{N}^{21}\\
			& \tilde{p}_{N}^{11} & \tilde{p}_{N}^{12}\\
			& \tilde{p}_{N}^{21} & \widetilde{P}_{N}^{22}\\
			& & & \tilde{\nu} I\\
			\widetilde{M}_{N}^{12} & & & & \widetilde{M}_{N}^{11}
		\end{pmatrix}
		(\bullet)^\top
	\end{align*}
	by $\widetilde{Q}_{N}$ and introduce the new notation
	\begin{align*}
		\begin{pmatrix}
			\tilde{f}_{N}^K & \widetilde{A}_{N}^K\\
			\tilde{g}_{N}^{1K} & \widetilde{C}_{N}^{1K}\\
			\tilde{g}_{N}^{2K} & \widetilde{C}_{N}^{2K}
		\end{pmatrix}
		:=
		\begin{pmatrix}
			f_{N} & A_{N} & B_{N}^1\\
			g_{N}^1 & C_{N}^1 & D_{N}^{11}\\
			g_{N}^2 & C_{N}^2 & D_{N}^{21}
		\end{pmatrix}
		\begin{pmatrix}
			\tilde{p}_{N}^{11} & \tilde{p}_{N}^{12}\\
			\tilde{p}_{N}^{21} & \widetilde{P}_{N}^{22}\\
			\tilde{k}_{N}^1 & \widetilde{K}_{N}^2
		\end{pmatrix}.
	\end{align*}
	Then the matrix inequalities \eqref{eq:robustCostInequality}, \eqref{eq:initialInequality1}, \eqref{eq:infiniteCostInequality} and \eqref{eq:constraintInequality1} are equivalent to the LMIs \eqref{eq:robustLinearCostInequality}, \eqref{eq:linearInitialInequality}, \eqref{eq:constraintInequality2} and
	\begin{align}
		\begin{pmatrix}
			\tilde{q}_{N}^{11} & \tilde{q}_{N}^{12} & \tilde{q}_{N}^{13} & \tilde{q}_{N}^{14} & \tilde{p}_{N}^{11} & \tilde{p}_{N}^{12}\\
			\tilde{q}_{N}^{21} & \widetilde{Q}_{N}^{22} & \widetilde{Q}_{N}^{23} & \widetilde{Q}_{N}^{24} & \tilde{f}_{N}^K & \widetilde{A}_{N}^K\\
			\tilde{q}_{N}^{31} & \widetilde{Q}_{N}^{32} & \widetilde{Q}_{N}^{33} & \widetilde{Q}_{N}^{34} & \tilde{g}_{N}^{1K} & \widetilde{C}_{N}^{1K}\\
			\tilde{q}_{N}^{41} & \widetilde{Q}_{N}^{42} & \widetilde{Q}_{N}^{43} & \widetilde{Q}_{N}^{44} & \tilde{g}_{N}^{2K} & \widetilde{C}_{N}^{2K}\\
			\tilde{p}_{N}^{11} & (\tilde{f}_{N}^K)^\top & (\tilde{g}_{N}^{1K})^\top & (\tilde{g}_{N}^{2K})^\top & \tilde{p}_{N}^{11} & \tilde{p}_{N}^{12}\\
			\tilde{p}_{N}^{21} & (\widetilde{A}_{N}^K)^\top & (\widetilde{C}_{N}^{1K})^\top & (\widetilde{C}_{N}^{2K})^\top & \tilde{p}_{N}^{21} & \widetilde{P}_{N}^{22}
		\end{pmatrix}\succeq 0.
		\label{eq:infiniteRobustLinearCostInequality}
	\end{align}
\end{theorem}

\begin{proof}
	The proof follows the same lines as the one of Theorem \ref{thm:RobustPeakGainControllerSynthesis} and is omitted.
\end{proof}

Using Theorem \ref{thm:infiniteRobustPeakGainControllerSynthesis} we now formulate a convex relaxation of \eqref{eq:robustinfiniteDP}. In this problem, we maximize over the variable $\tilde{\nu}$, whose inverse provides a robust upper bound on the cost in \eqref{eq:robustinfiniteDP}. The resulting convex program is
\begin{align}
	\maximize_{\widetilde{P}_k,\widetilde{K}_k,\tilde{\nu},Z,\widetilde{M}_k \in \calM'} ~&~ \tilde{\nu} \label{eq:relaxedRobustInfiniteDP}\\
	\mathrm{s.t.} ~&~ \eqref{eq:robustLinearCostInequality} \text{ for } k = 0,\ldots,N-1, \nonumber\\
	~&~ \eqref{eq:constraintInequality2} \text{ for } k=0,\ldots, N, \nonumber\\
	~&~ \eqref{eq:linearInitialInequality} \text{ and } \eqref{eq:infiniteRobustLinearCostInequality}. \nonumber
\end{align}

\section{Receding horizon control with uncertainty}
\label{sec:5}

Resulting from the controller synthesis problem \eqref{eq:relaxedRobustInfiniteDP} we can define the receding horizon controller
\begin{align*}
	\pi_{j}^{\text{MPC}}(\bar{x}_j) = \widetilde{K}_{0,j} \widetilde{P}_{0,j}^{-1} \begin{pmatrix}
		1\\ \bar{x}_j
	\end{pmatrix} = K_{0,j} \begin{pmatrix}
	1\\ \bar{x}_j
\end{pmatrix}
\end{align*}
where $\widetilde{K}_{0,j}$ and $\widetilde{P}_{0,j}$ are obtained by solving \eqref{eq:relaxedRobustInfiniteDP} for the problem data $(G_{k,j})_{k=0}^{N} := (G_{j+k})_{k=0}^{N}$. 
Our final result is the recursive feasibility, robust convergence and robust constraint satisfaction of $\pi_{j}^{\text{MPC}}$. We highlight that $\pi_{j}^{\text{MPC}}$ is defined simply as the solution to problem \eqref{eq:relaxedRobustInfiniteDP}, which is recursively feasible by default. This property is a consequence of the constraint \eqref{eq:infiniteRobustLinearCostInequality}, which can be interpreted as a terminal ingredient for $\pi_{j}^{\text{MPC}}$.

\begin{theorem}[Recursive feasibility and convergence]
	\label{thm:MPC}
	If \eqref{eq:relaxedRobustInfiniteDP} is feasible at $\bar{x}_0$, then \eqref{eq:relaxedRobustInfiniteDP} is feasible at all states $\bar{x}_j$ with $j\in \bbN$ defined by the closed loop
	\begin{align}
		\bar{x}_{j+1} &= f_j + A_j \bar{x}_j + B_j^1 \pi_j^{\text{MPC}}(\bar{x}_j) + B_j^2 w_j \label{eq:recedingHorizonDynamics}
	\end{align}
	for any realization of $(w_j)_{j=0}^\infty$ satisfying \eqref{eq:disturbanceConstraint}. Furthermore, the constraints \eqref{eq:controlConstraint} are robustly satisfied and the signal $(y_j)_{j=0}^\infty$ converges to zero.
\end{theorem}
\begin{proof}
	See Appendix \ref{app:3}.
\end{proof}

In Theorem \ref{thm:MPC}, we interpret the convergence of $(y_j)_{j=0}^\infty$ to zero as stability result given the fact that an appropriate choice of $(g_k^1, C_k^1, D_k^{11}, D_k^{12})$ yields $y_j^\top = \begin{pmatrix}
	\bar{x}_j^\top & u_j^\top
\end{pmatrix}$ and implies $\bar{x}_j \to 0$ and $u_j \to 0$ for $j \to \infty$.

\section{Numerical example}
\label{sec:6}

We adopt an example from \cite{chen2022robust}, where the LTI system
\begin{align*}
	x_{k+1} = \begin{pmatrix}
		1+\epsilon_1 & 0.15\\
		0.1 & 1
	\end{pmatrix}
	x_k
	+
	\begin{pmatrix}
		0.1\\
		1.1 + \epsilon_2
	\end{pmatrix} u_k
\end{align*}
with $x_k \in \bbR^2$, $u_k \in \bbR$ and time-varying uncertain parameters $\epsilon_1$ and $\epsilon_2$ is considered. The states and inputs are constrained by
\begin{align*}
	x_k^\top \begin{pmatrix}
		1 & 0\\
		0 & 0
	\end{pmatrix} x_k &\leq 8^2, & x_k^\top \begin{pmatrix}
	0 & 0\\
	0 & 1
\end{pmatrix} x_k &\leq 8^2, & u_k^2 &\leq 4^2.
\end{align*}
In our git repository \url{https://github.com/SphinxDG/ConstrainedRobustControl}, we derive $G_k$ and a family of multipliers $\calM$ that describe this system, its constraints, and the uncertainties in the parameters $\epsilon_1$ and $\epsilon_2$. As a result, we can test the feasibility of the infinite horizon robust optimal control problem \eqref{eq:robustinfiniteDP} at some initial condition $\bar{x}$ by attempting to solve the relaxation \eqref{eq:relaxedRobustInfiniteDP}. We can further determine the feasibility of \eqref{eq:robustinfiniteDP} at $\bar{x}$ exactly using a dynamic programming method described in \cite{grieder2003robust}.
Thus, to evaluate our relaxation \eqref{eq:relaxedRobustInfiniteDP}, we sample initial conditions $\bar{x}$ from the equidistant $10 \times 10$ grid $\calX_0$ at $[-7.9,7. 9]^2$ and compute the fraction of initial conditions certified as feasible by \eqref{eq:relaxedRobustInfiniteDP} over the truly feasible initial conditions determined by the dynamic programming method of \cite{grieder2003robust}. To solve \eqref{eq:relaxedRobustInfiniteDP} we use the solver Mosek \cite{mosek} and the parser CVXPY \cite{agrawal2018rewriting}. The fraction of initial conditions correctly classified as feasible is plotted in Figure \ref{fig:coverageComparison} for different prediction horizons $N$ over an uncertainty magnitude $\gamma \in [0.05,0.45]$ with $\epsilon_1 \in [-\gamma,\gamma]$ and $\epsilon_2 \in [-0.1,0.1]$. As expected, a longer prediction horizon monotonically increases the fraction of feasible initial conditions and provides an improvement over Kothares' method \cite{kothare1996robust}, which is obtained for $N = 0$.

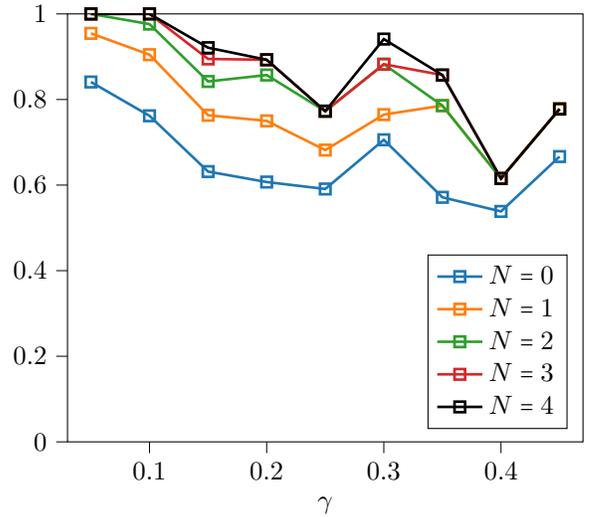
\begin{figure}
	\centering
	\input{coverageComparison.tex}
	\caption{This figure shows the fraction of initial conditions from $\calX_0 \cap \calF$, for which we found a solution $\tilde{\nu} > 0$ of \eqref{eq:relaxedRobustInfiniteDP}. Here, $\calF$ is the set of feasible initial conditions for the original problem \eqref{eq:robustinfiniteDP}. The parameter $\epsilon_1$ is assumed in $[-\gamma,\gamma]$ with $\gamma$ sampled from $\{0.05p \mid p = 1,\ldots,9\}$ and $\epsilon_2$ is assumed to be in $[-0.1,0.1]$. The plots show the fraction of feasible initial conditions for different prediction horizons $N$.}
	\label{fig:coverageComparison}
\end{figure}

As we mentioned, a similar experiment has been carried out in \cite{chen2022robust}, where an additional norm bounded disturbance has been considered. If we test the methods benchmarked in \cite{chen2022robust} on our modified setup and compare them to our solution, we obtain the results depicted in Figure \ref{fig:coverageComparison2}. We observe that SLSMPC achieves higher feasibility ratios, while our method performs at least as good as the best competitors  considered in \cite{chen2022robust}. Recall, however, that the method from \cite{chen2022robust} cannot be applied to systems in LFT-form directly and it is not shown to be recursively feasible in \cite{chen2022robust}. Furthermore, all methods considered in \cite{chen2022robust} were supplied with the exact feasible set of \eqref{eq:robustinfiniteDP} as terminal region, which can be expensive to compute in higher dimensions, whereas we did not make use of this set in \eqref{eq:relaxedRobustInfiniteDP}.

\begin{figure}
	\centering
	\input{coverageComparison2.tex}
	\caption{This figure compares the feasibility of \eqref{eq:relaxedRobustInfiniteDP} to other methods for robust MPC. The $y$-axis is the fraction of initial conditions from $\calX_0 \cap \calF$, for which feasibility of \eqref{eq:robustinfiniteDP} is certified. Here, $\calF$ is the set of feasible initial conditions for the original problem \eqref{eq:robustinfiniteDP}. The parameter $\epsilon_1$ is assumed in $[-\gamma,\gamma]$ with $\gamma$ sampled from $\{0.05p \mid p = 1,\ldots,9\}$ and $\epsilon_2$ is assumed to be in $[-0.1,0.1]$. The methods compared are our method (\ref{plt:thisPaper}), the tube based MPC methods \cite{langson2004robust} (\ref{plt:tubeA}), \cite{lorenzen2019robust} (\ref{plt:tubeB}), \cite{kohler2019linear} (\ref{plt:tubeC}), \cite{lu2019robust} (\ref{plt:tubeD}), the disturbance feedback methods \cite{bujarbaruah2022robust} (\ref{plt:lumpedDisturbance}), \cite{bujarbaruah2020robust} (\ref{plt:OfflineTightening}), and SLSMPC \cite{chen2022robust} (\ref{plt:SLSMPC}).}
	\label{fig:coverageComparison2}
\end{figure}
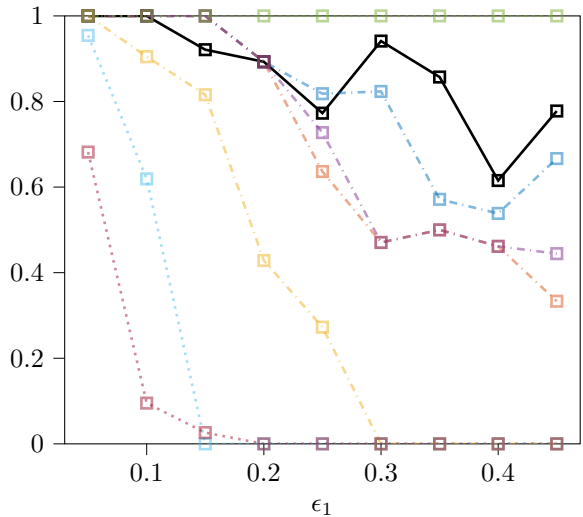

\section{Conclusion}
\label{sec:7}

In this article, we demonstrate how classical techniques from robust control theory can be exploited to solve dynamic programs with uncertainties and inequality constraints for time-varying linear systems. Since we invoke the standard relaxations from robust control for convexification, the solutions resulting from our method are inexact. We apply these solutions for MPC and highlight advantages over, e.g., tube-based MPC, which stems from the ability of our method to optimize feedback policies online and to incorporate flexible uncertainty descriptions in the form of LFT models. In future work, we aim to extend the method presented to nonlinear systems incorporating ideas from \cite{gramlich2022robust} and to reduce the computational burden due to solving LMIs making use of structure exploiting SDP solvers as in \cite{gramlich2023structure}.

\bibliographystyle{abbrv}
\bibliography{sources.bib}

\appendix

\subsection{Proof of Theorem \ref{thm:losslessCase}}
\label{app:1}
Let $(u_k)_{k=0}^{N-1}$ be a feasible input sequence for \eqref{eq:DP1}. Set the variables $(\widetilde{P}_k)_{k=0}^{N}$ and $(\widetilde{K}_k)_{k=0}^{N-1}$ equal to
\begin{align*}
	\widetilde{P}_k &= \frac{1}{c_k} \begin{pmatrix}
		1\\
		x_k
	\end{pmatrix} \begin{pmatrix}
	1\\ x_k
\end{pmatrix}^\top, & \widetilde{K}_k &= \frac{1}{c_k} u_k
	\begin{pmatrix}
		1\\ x_k
	\end{pmatrix}^\top
\end{align*}
with free parameters $(c_k)_{k=0}^N$, $c_k > 0$ for $k = 0,\ldots,N$ and $\tilde{\nu}$ and $Z$ to be determined. In this case, we obtain
\begin{align*}
	\begin{pmatrix}
		\tilde{f}_k^K & \widetilde{A}_k^K\\
		\tilde{g}_k^{1K} & \widetilde{C}_k^{1K}
	\end{pmatrix}
	&=
	\begin{pmatrix}
		f_k & A_k & B_k^1\\
		g_k^1 & C_k^1 & D_k^{11}\\
		g_{ki}^2 & C_{ki}^2 & D_{ki}^2
	\end{pmatrix}
	\begin{pmatrix}
		\tilde{p}_k^{11} & \tilde{p}_k^{12}\\
		\tilde{p}_k^{21} & \widetilde{P}_k^{22}\\
		\tilde{k}_k^1 & \widetilde{K}_k^2
	\end{pmatrix}\\
	&=
	\frac{1}{c_k}
	\begin{pmatrix}
		x_{k+1}\\
		y_k\\
		v_{ki}
	\end{pmatrix}
	\begin{pmatrix}
		1\\
		x_k
	\end{pmatrix}^\top .
\end{align*}
For these values, the matrix in the inequality \eqref{eq:linearCostInequality} reads
\begin{align*}
	\begin{pmatrix}
		\frac{1}{c_{k+1}} \begin{pmatrix}
			1\\
			x_{k+1}
		\end{pmatrix}
		\begin{pmatrix}
			1\\
			x_{k+1}
		\end{pmatrix}^\top & \begin{pmatrix}
			0\\
			0
		\end{pmatrix}
		&
		\frac{1}{c_k} \begin{pmatrix}
			1\\
			x_{k+1}
		\end{pmatrix}
		\begin{pmatrix}
			1\\
			x_{k}
		\end{pmatrix}^\top\\
		\begin{pmatrix}
			0 & 0
		\end{pmatrix} & I & \frac{1}{c_k} y_k \begin{pmatrix}
			1 & x_k^\top
		\end{pmatrix}\\
		\frac{1}{c_k} \begin{pmatrix}
			1\\
			x_{k}
		\end{pmatrix}
		\begin{pmatrix}
			1\\
			x_{k+1}
		\end{pmatrix}^\top &
		\frac{1}{c_k} \begin{pmatrix}
			1\\
			x_{k}
		\end{pmatrix} y_k^\top &
		\frac{1}{c_k} \begin{pmatrix}
			1\\ x_k
		\end{pmatrix} \begin{pmatrix}
		1\\ x_k
	\end{pmatrix}^\top
	\end{pmatrix}
\end{align*}
and must be positive semi-definite. Left- and right-multiplying this matrix with an arbitrary vector $\begin{pmatrix}
	r_1^\top & r_2^\top & r_3^\top
\end{pmatrix}$ and its transpose, yields the condition
\begin{align*}
	\begin{pmatrix}
		h_1\\ 1\\ h_2
	\end{pmatrix}^\top
	\begin{pmatrix}
		\frac{1}{c_{k+1}} & 0
		&
		\frac{1}{c_k} \\
		0 & r_2^\top r_2 & \frac{1}{c_k} r_2^\top y_k\\
		\frac{1}{c_k} &
		\frac{1}{c_k} y_k^\top r_2 &
		\frac{1}{c_k}
	\end{pmatrix}
	\begin{pmatrix}
		h_1\\ 1\\ h_2
	\end{pmatrix} \geq 0,
\end{align*}
where $h_1 = \begin{pmatrix}
	1 & x_{k+1}^\top
\end{pmatrix}r_1$ and $h_2 = \begin{pmatrix}
1 & x_{k}^\top
\end{pmatrix}r_3$. This constraint holds for all $h_1, h_2 \in \bbR$ if and only if the matrix in the center is positive semi-definite. By the Schur complement, this is the case if and only if
\begin{align}
	\frac{1}{c_k} - \frac{c_{k+1}}{c_k^2} - \frac{1}{c_k^2} \frac{r_2^\top y_k y_k^\top r_2}{r_2^\top r_2} &\geq 0 &\text{if } r_k \neq 0, \label{eq:rNonZero}\\
	\frac{1}{c_k} - \frac{c_{k+1}}{c_k^2}  &\geq 0 &\text{if } r_k = 0. \label{eq:rZero}
\end{align}
Here, $\frac{r_2^\top y_k y_k^\top r_2}{r_2^\top r_2}$ is bounded by the Rayleigh coefficient $y_k^\top y_k$. Hence, multiplying by $c_k^2$ yields the requirement
\begin{align*}
	c_k \geq c_{k+1} + y_k^\top y_k,
\end{align*}
which implies both \eqref{eq:rNonZero} and \eqref{eq:rZero}. Applying a similar left/right multiplication to \eqref{eq:constraintInequality2} results in the condition
\begin{align*}
	\frac{1}{c_k} v_{ki}^\top v_{ki} \leq \tilde{\nu}.
\end{align*}
Next, we consider the constraint \eqref{eq:linearInitialInequality}. The term $\sqrt{\Sigma_0}$ equals $\frac{1}{\sqrt{1 + \bar{x}^\top \bar{x}}} \begin{pmatrix}
	1 & \bar{x}^\top
\end{pmatrix}^\top \begin{pmatrix}
1 & \bar{x}^\top
\end{pmatrix}$. Further, we choose $Z = \frac{c_0 \tilde{\nu}^{2}}{1 + \bar{x}^\top \bar{x}} \begin{pmatrix}
1 & \bar{x}^\top
\end{pmatrix}^\top \begin{pmatrix}
1 & \bar{x}^\top
\end{pmatrix}$. With these choices, \eqref{eq:linearInitialInequality} is composed of the constraints $\trace Z = c_0 \tilde{\nu}^2 \leq \tilde{\nu}$ and
\begin{align*}
	\begin{pmatrix}
		\frac{1}{c_0}
		\begin{pmatrix}
			1 \\ \bar{x}
		\end{pmatrix} \begin{pmatrix}
			1 \\ \bar{x}
		\end{pmatrix}^\top & \frac{\tilde{\nu}}{\sqrt{1 + \bar{x}^\top \bar{x}}} \begin{pmatrix}
			1 \\ \bar{x}
		\end{pmatrix} \begin{pmatrix}
		1 \\ \bar{x}
	\end{pmatrix}^\top\\
	\frac{\tilde{\nu}}{\sqrt{1 + \bar{x}^\top \bar{x}}} \begin{pmatrix}
		1 \\ \bar{x}
	\end{pmatrix} \begin{pmatrix}
		1 \\ \bar{x}
	\end{pmatrix}^\top & \frac{c_0 \tilde{\nu}^{2}}{1 + \bar{x}^\top \bar{x}} \begin{pmatrix}
	1 \\ \bar{x}
\end{pmatrix} \begin{pmatrix}
1 \\ \bar{x}
\end{pmatrix}^\top
	\end{pmatrix} \succeq 0.
\end{align*}
Of these requirements, the matrix inequality is always satisfied and the trace condition is equivalent to $c_0 \leq \frac{1}{\tilde{\nu}}$.

Finally, we consider \eqref{eq:linearTerminalInequality}. This constraint reads
\begin{align*}
	\frac{1}{c_{N}} \begin{pmatrix}
		1\\x_{N}
	\end{pmatrix}
	\begin{pmatrix}
		1\\
		x_{N}
	\end{pmatrix}^\top \preceq P_f^{-1} \Leftrightarrow c_{N} \geq \begin{pmatrix}
		1\\x_{N}
	\end{pmatrix}^\top P_f
	\begin{pmatrix}
	1\\
	x_{N}
	\end{pmatrix},
\end{align*}
where the equivalence is established by applying the Schur complement twice. We summarize the constraints on the sequence $(c_k)_{k=0}^N$ and the variable $\tilde{\nu}$:
\begin{subequations}
	\label{eq:c_conditions}
	\begin{align}
		c_0 & \leq \tilde{\nu}^{-1}, \label{eq:initial_c}\\
		c_k & \geq \tilde{\nu}^{-1} v_{ki}^\top v_{ki}, & k=0,\ldots, N,~ i = 1,\ldots,s, \label{eq:constraint_c}\\
		c_k & \geq y_k^\top y_k + c_{k+1}, & k = 0,\ldots,N, \label{eq:dissipativity_c}\\
		c_{N} & \geq \begin{pmatrix}
			1\\x_{N}
		\end{pmatrix}^\top P_f
		\begin{pmatrix}
			1\\
			x_{N}
		\end{pmatrix}. \label{eq:terminal_c}
	\end{align}
\end{subequations}
By choosing $\varepsilon = 1- \max_{k,i} v_{ki}^\top v_{ki} > 0$ and for $j = 0,\ldots,N$
\begin{align*}
c_j &= \sum_{j = j}^{N-1} y_k^\top y_k + 2\max \left\{ \varepsilon^{-1}\sum_{k=0}^{N-1} y_k^\top y_k , \begin{pmatrix}
	1\\x_{N}
\end{pmatrix}^\top P_f
\begin{pmatrix}
	1\\
	x_{N}
\end{pmatrix} \right\},\\
\tilde{\nu}^{-1} &= (1-\varepsilon)^{-1} c_N,
\end{align*}
\eqref{eq:initial_c}-\eqref{eq:terminal_c} are satisfied, with the trace condition in \eqref{eq:initial_c} strictly satisfied.
Hence, we see that a strictly feasible candidate solution to \eqref{eq:DP1} enables the construction of a non-strictly feasible candidate solution to \eqref{eq:relaxedDP1}.

Next, we construct a suitable perturbation of $(\widetilde{P}_k)_{k=0}^{N}$, $(\widetilde{K}_k)_{k=0}^{N-1}$, $Z$ to strictly satisfy the constraints of \eqref{eq:relaxedDP1}.

To this end, consider strict versions of the matrix inequalities \eqref{eq:bellmanInequality1} and \eqref{eq:terminalInequality1}, which can be denoted
\begin{align*}
	\begin{pmatrix}
		p_k^{11} & p_k^{12}\\
		p_k^{21} & P_k^{22}
	\end{pmatrix}
	&\succ 
	(\bullet)^\top
	\begin{pmatrix}
		p_{k+1}^{11} & p_{k+1}^{12} & 0\\
		p_{k+1}^{21} & P_{k+1}^{22} & 0\\
		0 & 0 & I
	\end{pmatrix}
	\begin{pmatrix}
		1 & 0\\
		f_k^K & A_k^K\\
		g_k^{1K} & C_k^{1K}
	\end{pmatrix},\\
	P_N &\succ P_f.
\end{align*}
Determining a solution $(\hat{P}_k)_{k=0}^N$ to these inequalities for the controller $\hat{K}_k := 0, k = 0,\ldots ,N-1$ is always possible in a Riccati-like backward sweep. Since this solution consists of positive definite matrices $\hat{P}_k$, it can be multiplied by a sufficiently large scalar, such that \eqref{eq:constraintInequality1} holds strictly for $k = 0,\ldots, N-1$ and $i = 1,\ldots,s$ for $\nu = \tilde{\nu}^{-1}$ and $\tilde{\nu}$ chosen as in the privious step. The constraint \eqref{eq:initialInequality1}, however, will not hold for this solution $(P_k)_{k=0}^N = (\hat{P}_k)_{k=0}^N$, $(K_k)_{k=0}^{N-1} = (\hat{K}_k)_{k=0}^{N-1}$. Now, following the lines of the proof of Theorem \ref{thm:PeakGainControllerSynthesis} reveals that the variables
\begin{align*}
	\widehat{\widetilde{P}}_k := \widehat{P}_k^{-1},~~ \widehat{\widetilde{K}}_k := \widehat{K}_k\widehat{P}_k^{-1},~~ \hat{\tilde{\nu}} := \tilde{\nu},~~ \widehat{Z} := 2\nu^{-2} \sqrt{\Sigma_0} \widehat{P}_0 \sqrt{\Sigma_0}
\end{align*}
satisfy all conditions in \eqref{eq:relaxedDP1} strictly except for $\trace Z \leq \tilde{\nu}$. Now, consider the convex combinations
\begin{align*}
	\widetilde{P}_k(\alpha) &= \alpha \widehat{\widetilde{P}}_k + (1-\alpha) \widetilde{P}_k, & \widetilde{K}_k(\alpha) \alpha &= \widehat{\widetilde{K}}_k + (1-\alpha) \widetilde{K}_k,\\
	Z(\alpha) &= \alpha \widehat{Z} + (1-\alpha) Z, & \tilde{\nu}(\alpha) &= \alpha \hat{\tilde{\nu}} + (1-\alpha) \tilde{\nu}.
\end{align*}
Since the constraints in \eqref{eq:relaxedDP1} are LMI constraints, these convex combination satisfy all constraints strictly  for any $\alpha \in ]0,1[$ except for $\trace Z(\alpha) \leq \tilde{\nu}(\alpha)$. However, since $Z = Z(0)$ and $\tilde{\nu} = \tilde{\nu}(0)$ satisfy this constraint strictly, there exists a sufficently small $\alpha > 0$, such that also this constraint is strictly satisfied by the convex combination.

These arguments show how a strictly feasible candidate solution to \eqref{eq:DP1} enables the construction of a strictly feasible candidate solution to \eqref{eq:relaxedDP1}. For the other direction, we simply construct the controller $\pi_k(x_k) := \widetilde{K}_k \widetilde{P}_k^{-1} \begin{pmatrix}
	1 & x_k^\top
\end{pmatrix}^\top$ and perform a forward simulation. {\hfill$\blacksquare$}

\subsection{On the optimal value of \eqref{eq:DP1} and \eqref{eq:relaxedDP1}}
\label{app:2}
In this section, we show that, if \eqref{eq:DP1} is strictly feasible, then the optimal cost of \eqref{eq:DP1} can be approximated arbitrarily closely by replacing $P_f$ with $P_f + t e_1 e_1^\top$ and considering $\tilde{\nu}^{-1} - t$ as the new optimal value of \eqref{eq:relaxedDP1} for a sufficiently large value of $t$. Notice that this leaves the optimal control problem unchanged, since adding $t e_1 e_1^\top$ to $P_f$ corresponds to adding the constant $t$ to the terminal cost and substracting $t$ from $\tilde{\nu}$ removes this constant offset from the final value of the cost.

To understand why this works, notice that a solution to \eqref{eq:c_conditions} is given by the choice
\begin{align*}
	c_j &= \sum_{j = j}^{N-1} y_k^\top y_k + \max \left\{ \varepsilon^{-1}\sum_{k=0}^{N-1} y_k^\top y_k , \begin{pmatrix}
		1\\x_{N}
	\end{pmatrix}^\top P_f
	\begin{pmatrix}
		1\\
		x_{N}
	\end{pmatrix} \right\},\\
	\tilde{\nu}^{-1} &=  \sum_{j = 0}^{N-1} y_k^\top y_k + \max \left\{ \varepsilon^{-1}\sum_{k=0}^{N-1} y_k^\top y_k , \begin{pmatrix}
		1\\x_{N}
	\end{pmatrix}^\top P_f
	\begin{pmatrix}
		1\\
		x_{N}
	\end{pmatrix} \right\},
\end{align*}
This choice for $(c_k)_{k=0}^{N}$ and $\tilde{\nu}$ is constructed from any solution of \eqref{eq:DP1} with $\varepsilon = 1- \max_{k,i} v_{ki}^\top v_{ki} > 0$ for all $k = 0,\ldots,N-1$, $i = 1,\ldots,s$ and enables the construction of a solution to problem \eqref{eq:relaxedDP1} along the lines of Appendix \ref{app:1}. Recall that the cost of \eqref{eq:relaxedDP1} is given by $\tilde{\nu}^{-1}$. When replacing $P_f$ by $P_f + t e_1 e_1^\top$, this yields
\begin{align*}
	\tilde{\nu}^{-1} &=  \sum_{j = 0}^{N-1} y_k^\top y_k + \max \left\{ \varepsilon^{-1}\sum_{k=0}^{N-1} y_k^\top y_k , t+\begin{pmatrix}
		1\\x_{N}
	\end{pmatrix}^\top P_f
	\begin{pmatrix}
		1\\
		x_{N}
	\end{pmatrix} \right\}
\end{align*}
and, consequently, for a sufficiently large value of $t$, $\tilde{\nu}^{-1} - t$ equals the optimal cost of the considered trajectory.

If \eqref{eq:DP1} is strictly feasible, then any optimal trajectory of \eqref{eq:DP1} can be approximated arbitrarily closely by some strictly feasible trajectory to which the above procedure applies. Consequently, for $t \to \infty$, $\tilde{\nu}^{-1} - t$ approaches the optimal cost of \eqref{eq:DP1}.

\subsection{Proof of Theorem \ref{thm:MPC}}
\label{app:3}

Denote \eqref{eq:robustCostInequality} by $F(G_k,P_k,P_{k+1},K_k,M_k) \preceq 0$ and
\eqref{eq:constraintInequality1} by $H(G_k,P_k,K_k,\nu) \preceq 0$. Then determining the input $u_j = \pi_j^{\text{MPC}}(\bar{x}_j)$ consists of finding a minimal $\nu_j$ as well as sequences $(P_{k,j})_{k=0}^{N}$, $(K_{k,j})_{k=0}^{N}$ satisfying
\begin{subequations}
\begin{align}
	0 & \succeq F(G_{k,j},P_{k,j},P_{k+1,j},K_{k,j},M_{k,j}) \quad k = 0,\ldots,N-1 \label{eq:recedingCostIneq}\\
	0 &\succeq H(G_{k,j},P_{k,j},K_{k,j},\nu_j) \qquad k = 0,\ldots,N \label{eq:recedingConstrIneq}\\
	0 & \geq \begin{pmatrix}
		1 & \bar{x}_j^\top
	\end{pmatrix} P_{0,j} \begin{pmatrix}
		1 & \bar{x}_j^\top
	\end{pmatrix}^\top - \nu_j \label{eq:recedingInitialIneq}\\
	0 & \succeq F(G_{N,j},P_{N,j},P_{N,j},K_{N,j},M_{N,j}) \label{eq:recedingTerminalIneq}
\end{align}
\end{subequations}
using the convexification \eqref{eq:relaxedRobustInfiniteDP} and setting $u_j = K_{0,j} \begin{pmatrix}
	1 & \bar{x}_j^\top
\end{pmatrix}^\top$. We show now that if this problem is feasible at time $j \in \bbN$ and $\bar{x}_{j+1}$ is determined through the system dynamics \eqref{eq:recedingHorizonDynamics}, where $w_j$ is some disturbance satisfying \eqref{eq:disturbanceConstraint}, then this problem is feasible at time $j+1$.

To this end, we define the candidate solution 
\begin{align*}
	P_{k,j+1}' &= P_{k+1,j}, & K_{k,1}' &= K_{k+1,j}, & k = 0,\ldots, N-1\\
	P_{N,j+1}' &= P_{N,j}, & K_{N,j+1}' &= K_{N,j},
\end{align*}
$\nu_{j+1}' = \nu_j - y_j^\top y_j$ at time $j+1$. Note that the problem data $G_{k,j+1} = G_{k+1,j}$ is shifted in the same fashion. For this candidate solution, we have
\begin{align*}
	\nu_{j+1}' = \nu_j - y_j^\top y_j & \geq 
	\begin{pmatrix}
		1 & \bar{x}_j^\top
	\end{pmatrix} P_{0,j} \begin{pmatrix}
		1 & \bar{x}_j^\top
	\end{pmatrix}^\top - y_j^\top y_j\\
	& \overset{\eqref{eq:robustCostBound}}{\geq}
	\begin{pmatrix}
		1 & \bar{x}_{j+1}^\top
	\end{pmatrix} P_{1,j} \begin{pmatrix}
		1 & \bar{x}_{j+1}^\top
	\end{pmatrix}^\top\\
	&= \begin{pmatrix}
		1 & \bar{x}_{j+1}^\top
	\end{pmatrix} P_{0,j+1}' \begin{pmatrix}
		1 & \bar{x}_{j+1}^\top  \end{pmatrix}^\top
\end{align*}
where \eqref{eq:robustCostBound} holds due to \eqref{eq:recedingCostIneq} and Proposition \ref{prop:PreliminaryRobustInequaities}. This shows \eqref{eq:recedingInitialIneq} at time $j + 1$. Next, \eqref{eq:recedingCostIneq} holds trivially for $k = 0,\ldots, N -2$ at time $j+1$. For $k = N - 1$, on the other hand, $P_{k,j+1}' = P_{k+1,j+1}' = P_{N,j}$, such that \eqref{eq:recedingCostIneq} for $k = N-1$ at time $j+1$ is implied by \eqref{eq:recedingTerminalIneq} at time $j$.
The constraints \eqref{eq:recedingConstrIneq} for the candidate solution at time $j+1$ are the same as at time $j$ for the original solution except for the first constraint getting replaced by an additional copy of the last constraint and $\nu_{j+1}'$ replacing $\nu_j$. However, an inspection of \eqref{eq:constraintInequality1} reveals that replacing $\nu_j$ by $\nu_{j+1}' \leq \nu_j$ cannot cause constraint violation. Finally, \eqref{eq:recedingTerminalIneq} also remains feasible for the candidate solution due to the time invariance of all problem parameters for $k \geq N$. In conclusion, we constructed a candidate solution for the problem at time $j+1$, which is feasible with $\nu_{j+1}' = \nu_j - y_j^\top y_j$. Consequently, evaluation of $\pi_{j+1}^{\mathrm{MPC}}(\bar{x}_{j+1})$ might produce a different solution with potentially smaller $\nu_{j+1}$, but can never be infeasible. For this reason, $\nu_{j+1} \leq \nu_j - y_j^\top y_j$ must hold implying
\begin{align*}
	\nu_0 \geq \sum_{j=0}^\infty y_j^\top y_j.
\end{align*}
Consequently, $y_j$ must converge to zero under this controller. Finally, the constraint \eqref{eq:recedingConstrIneq} implies that constraints are always satisfied by this controller. {\hfill$\blacksquare$}

\end{document}

%% file: coverageComparison.tex
\begin{tikzpicture}

\definecolor{crimson2143940}{RGB}{214,39,40}
\definecolor{darkgray176}{RGB}{176,176,176}
\definecolor{darkorange25512714}{RGB}{255,127,14}
\definecolor{forestgreen4416044}{RGB}{44,160,44}
\definecolor{mediumpurple148103189}{RGB}{148,103,189}
\definecolor{steelblue31119180}{RGB}{31,119,180}

\begin{axis}[
tick align=outside,
tick pos=left,
x grid style={darkgray176},
xlabel={$\gamma$},
xmin=0.03, xmax=0.47,
xtick style={color=black},
y grid style={darkgray176},
ymin=0, ymax=1,
ytick style={color=black},
legend pos={south east}
]
\addplot [line width=1.0pt, steelblue31119180, mark=square, mark options={solid, steelblue31119180}]
table {%
0.05 0.840909090909091
0.1 0.761904761904762
0.15 0.631578947368421
0.2 0.607142857142857
0.25 0.590909090909091
0.3 0.705882352941177
0.35 0.571428571428571
0.4 0.538461538461538
0.45 0.666666666666667
};
\addlegendentry{$N = 0$}
\addplot [line width=1.0pt, darkorange25512714, mark=square, mark options={solid, darkorange25512714}]
table {%
0.05 0.954545454545455
0.1 0.904761904761905
0.15 0.763157894736842
0.2 0.75
0.25 0.681818181818182
0.3 0.764705882352941
0.35 0.785714285714286
0.4 0.615384615384615
0.45 0.777777777777778
};
\addlegendentry{$N = 1$}
\addplot [line width=1.0pt, forestgreen4416044, mark=square, mark options={solid, forestgreen4416044}]
table {%
0.05 1
0.1 0.976190476190476
0.15 0.842105263157895
0.2 0.857142857142857
0.25 0.772727272727273
0.3 0.882352941176471
0.35 0.785714285714286
0.4 0.615384615384615
0.45 0.777777777777778
};
\addlegendentry{$N = 2$}
\addplot [line width=1.0pt, crimson2143940, mark=square, mark options={solid, crimson2143940}]
table {%
0.05 1
0.1 1
0.15 0.894736842105263
0.2 0.892857142857143
0.25 0.772727272727273
0.3 0.882352941176471
0.35 0.857142857142857
0.4 0.615384615384615
0.45 0.777777777777778
};
\addlegendentry{$N = 3$}
\addplot [line width=1.0pt, black, mark=square, mark options={solid, black}]
table {%
0.05 1
0.1 1
0.15 0.921052631578947
0.2 0.892857142857143
0.25 0.772727272727273
0.3 0.941176470588235
0.35 0.857142857142857
0.4 0.615384615384615
0.45 0.777777777777778
};
\addlegendentry{$N = 4$}
\end{axis}

\end{tikzpicture}

%% file: coverageComparison2.tex
\begin{tikzpicture}
	
	\definecolor{crimson2143940}{RGB}{214,39,40}
	\definecolor{darkgray176}{RGB}{176,176,176}
	\definecolor{darkorange25512714}{RGB}{255,127,14}
	\definecolor{forestgreen4416044}{RGB}{44,160,44}
	\definecolor{mediumpurple148103189}{RGB}{148,103,189}
	\definecolor{steelblue31119180}{RGB}{31,119,180}

	\definecolor{mycolor1}{rgb}{0.00000,0.44700,0.74100}%
	\definecolor{mycolor2}{rgb}{0.85000,0.32500,0.09800}%
	\definecolor{mycolor3}{rgb}{0.92900,0.69400,0.12500}%
	\definecolor{mycolor4}{rgb}{0.49400,0.18400,0.55600}%
	\definecolor{mycolor5}{rgb}{0.46600,0.67400,0.18800}%
	\definecolor{mycolor6}{rgb}{0.30100,0.74500,0.93300}%
	\definecolor{mycolor7}{rgb}{0.63500,0.07800,0.18400}%
	
	\begin{axis}[
		tick align=outside,
		tick pos=left,
		x grid style={darkgray176},
		xlabel={$\epsilon_1$},
		xmin=0.03, xmax=0.47,
		xtick style={color=black},
		y grid style={darkgray176},
		ymin=0, ymax=1,
		ytick style={color=black},
		]
		\addplot [line width=1.0pt, black, mark=square, mark options={solid, black}]
		table {%
			0.05 1
			0.1 1
			0.15 0.921052631578947
			0.2 0.892857142857143
			0.25 0.772727272727273
			0.3 0.941176470588235
			0.35 0.857142857142857
			0.4 0.615384615384615
			0.45 0.777777777777778
		}; \label{plt:thisPaper}
		\addplot [color=mycolor1, dashdotted, line width=1.0pt, mark=square, mark options={solid, mycolor1},opacity=0.5]
		table[row sep=crcr]{%
			0.05	1\\
			0.1	1\\
			0.15	1\\
			0.2	0.892857142857143\\
			0.25	0.818181818181818\\
			0.3	0.823529411764706\\
			0.35	0.571428571428571\\
			0.4	0.538461538461538\\
			0.45	0.666666666666667\\
		}; \label{plt:tubeA}
		
		\addplot [color=mycolor2, dashdotted, line width=1.0pt, mark=square, mark options={solid, mycolor2},opacity=0.5]
		table[row sep=crcr]{%
			0.05	1\\
			0.1	1\\
			0.15	1\\
			0.2	0.892857142857143\\
			0.25	0.636363636363636\\
			0.3	0.470588235294118\\
			0.35	0.5\\
			0.4	0.461538461538462\\
			0.45	0.333333333333333\\
		}; \label{plt:tubeB}
		
		\addplot [color=mycolor3, dashdotted, line width=1.0pt, mark=square, mark options={solid, mycolor3},opacity=0.5]
		table[row sep=crcr]{%
			0.05	1\\
			0.1	0.904761904761905\\
			0.15	0.815789473684211\\
			0.2	0.428571428571429\\
			0.25	0.272727272727273\\
			0.3	0\\
			0.35	0\\
			0.4	0\\
			0.45	0\\
		}; \label{plt:tubeC}
		
		\addplot [color=mycolor4, dashdotted, line width=1.0pt, mark=square, mark options={solid, mycolor4},opacity=0.5]
		table[row sep=crcr]{%
			0.05	1\\
			0.1	1\\
			0.15	1\\
			0.2	0.892857142857143\\
			0.25	0.727272727272727\\
			0.3	0.470588235294118\\
			0.35	0.5\\
			0.4	0.461538461538462\\
			0.45	0.444444444444444\\
		}; \label{plt:tubeD}
		
		\addplot [color=mycolor5, dotted, line width=1.0pt, mark=square, mark options={solid, mycolor5},opacity=0.5]
		table[row sep=crcr]{%
			0.05	1\\
			0.1	1\\
			0.15	1\\
			0.2	1\\
			0.25	1\\
			0.3	1\\
			0.35	1\\
			0.4	1\\
			0.45	1\\
		}; \label{plt:SLSMPC}
		
		\addplot [color=mycolor6, dotted, line width=1.0pt, mark=square, mark options={solid, mycolor6},opacity=0.5]
		table[row sep=crcr]{%
			0.05	0.954545454545455\\
			0.1	0.619047619047619\\
			0.15	0\\
			0.2	0\\
			0.25	0\\
			0.3	0\\
			0.35	0\\
			0.4	0\\
			0.45	0\\
		}; \label{plt:lumpedDisturbance}
		
		\addplot [color=mycolor7, dotted, line width=1.0pt, mark=square, mark options={solid, mycolor7},opacity=0.5]
		table[row sep=crcr]{%
			0.05	0.681818181818182\\
			0.1	0.0952380952380952\\
			0.15	0.0263157894736842\\
			0.2	0\\
			0.25	0\\
			0.3	0\\
			0.35	0\\
			0.4	0\\
			0.45	0\\
		}; \label{plt:OfflineTightening}
	\end{axis}
	
\end{tikzpicture}